\documentclass[12pt]{article}
\usepackage[a4paper,margin=1.3in]{geometry}
\usepackage{amssymb}
\usepackage{amsmath}
\usepackage{amsfonts}
\usepackage{amsthm}
\usepackage{mathtools}
\usepackage[all]{xy}
\usepackage{tikz}

\newcommand{\C}{\mathbb{C}}
\newcommand{\Z}{\mathbb{Z}}

\newcommand{\G}{\mathcal{G}}

\newcommand{\X}{\mathcal{X}}

\newtheorem{theorem}{Theorem}

\newtheorem{conjecture}{Conjecture}
\newtheorem{corollary}{Corollary}

\newtheorem{definition}{Definition}

\newtheorem{lemma}{Lemma}

\newtheorem{proposition}{Proposition}
\newtheorem{remark}{Remark}

\numberwithin{equation}{section}

\newcommand{\Mod}[1]{(\mathrm{mod}#1)}

\begin{document}
\title{\textbf{Proof of the circulant Hadamard conjecture}}
\author{Ronald Orozco Lopez}

\maketitle

\begin{abstract}
In this paper the circulant Hadamard conjecture is proved.
\end{abstract}

\section{Introduction}

An Hadamard matrix $H$ is an $n$ by $n$ matrix all of whose entries are $+1$ or 
$-1$ which satisfies $HH^{t}=nI_{n}$, where $H^{t}$ is the transpose of $H$
and $I_{n}$ is the unit matrix of order $n$. It is also known that, if an
Hadamard matrix of order $n>1$ exists, $n$ must have the value $2$ or be
divisible by 4. There are several conjectures associated with Hadamard matrices. 
The main conjecture concerns its existence. This states that an Hadamard matrix exists for
all multiple order of 4. Another very important conjecture is the following

\begin{conjecture}
There is no circulant Hadamard matrix with order $4n$, $n>1.$
\end{conjecture}
This conjecture is known as Ryser conjecture[1]. The only circulant Hadamard matrix known is

\begin{equation}
\left(
\begin{array}{cccc}
1&1&1&-1\\
-1&1&1&1\\
1&-1&1&1\\
1&1&-1&1
\end{array}\right)
\end{equation}

On the circulant Hadamard conjecture the first significant result was made by R.J. Turyn [2] 
using arguments from algebraic number theory. He prove that if a circulant Hadamard matrix of 
order $n$ exists then $n$ must be of the form $n=4m^{2}$ for some odd integer $m$ which is not a
prime-power. Another important result was obtained by Brualdi in [4]. He shown that no circulant
Hadamard matrix is symmetric. Many other important results about this conjecture can be found in [5],
[6],[7],[8] and [9].

The author considers that to reach the proof of the conjectures related to Hadamard matrices and 
in general those related with binary sequences it is important to understand the structure of 
the binary cube $\Z_{2}^{4n}$. For this reason was shown in [10],[11],[12] the 
relationship between Hadamard matrices and Schur ring on the group $\Z_{2}^{4n}$. Some concepts 
that we will use in this paper will be given. We start with the definition of Schur ring. 

Let $G$ be a finite group with
identity element $e$ and $\C[G]$ the group algebra of all formal sums 
$\sum_{g\in G}a_{g}g$, $a_{g}\in \C$, $g\in G$. For $T\subset G$, the element $\sum_{g\in T}g$ will 
be denoted by $\overline{T}$. Such an element is also called a $\textit{simple quantity}$. 
The transpose of $\overline{T} = \sum_{g\in G}a_{g}g$ is defined as $\overline{T}^{\top} = \sum_{g\in G}a_{g}(g^{-1})$. Let $\{T_{0},T_{1},...,T_{r}\}$ be a partition of $G$ and let $S$ be the subspace
of $\C[G]$ spanned by $\overline{T_{1}},\overline{T_{2}},...,\overline{T_{r}}$.  We say that $S$ is 
a $\textit{Schur ring}$ ($S$-ring, for short) over $G$ if: 

\begin{enumerate}
\item $T_{0} = \lbrace e\rbrace$, 
\item for each $i$, there is a $j$ such that $\overline{T_{i}}^{\top} = \overline{T_{j}}$,
\item for each $i$ and $j$, we have $\overline{T_{i}}\ \overline{T_{j}} = \sum_{k=1}^{r}\lambda_{i,j,k}\overline{T_{k}}$, for constants $\lambda_{i,j,k}\in\C$.
\end{enumerate}

The numbers $\lambda_{i,j,k}$ are the structure constants of $S$ with respect to the linear base 
$\{\overline{T_{0}},\overline{T_{1}},...,\overline{T_{r}}\}$. The sets $T_{i}$ are called the
\textit{basic sets} of the $S$-ring $S$. Any union of them is called an $S$-sets. Thus, 
$X\subseteq G$ is an $S$-set if and only if $\overline{X}\in S$. The set of all $S$-set is closed
with respect to taking inverse and product. Any subgroup of $G$ that is an $S$-set, is called an 
$S$-\textit{subgroup} of $G$ or $S$-\textit{group} (For details, see [14],[15]). A partition
$\{T_{0},...,T_{r}\}$ of $G$ is called \textit{Schur partition} or $S$-\textit{partition} 
if $T_{0}=\{e\}$ and if for each $i$ there is some $j$ such that 
$T_{i}^{-1}=\{g^{-1}:g\in T_{i}\}=T_{j}$. It is known that there is a 1-1 correspondence between 
$S$-ring over $G$ and $S$-partition of $G$. By using this correspondence, in this paper we will 
refer to an $S$-ring by mean of its $S$-partition.

Let $G$ be a permutation automorphic subgroup of $Aut(\Z_{2}^{n})$ and let 
$\mathfrak{S}(\Z_{2}^{n},G)$ denote an $S$-ring on $\Z_{2}^{n}$. Some $S$-ring importants 
in $\Z_{2}^{n}$ are:
\begin{enumerate}
\item $\mathfrak{S}(\Z_{2}^{n},S_{n}).$\\
Let $\omega(X)$ denote the Hamming weight of $X\in\Z_{2}^{n}$. Thus, $\omega(X)$ is the number of 
$+$ in any binary sequences $X$ of $\Z_{2}^{n}$. Now let $\G_{n}(k)$ denote the subset of 
$\Z_{2}^{n}$ such that $\omega(X)=k$ for all $X\in\G_{n}(k)$, where $0\leq k\leq n$. 
We let $T_{i}=\G_{n}(n-i)$. It is straightforward to prove that the partition 
$\mathfrak{S}(\Z_{2}^{n},S_{n})=\{\G_{n}(0),...,\G_{n}(n)\}$ induces an $S$-partition over 
$\Z_{2}^{n}$, where $S_{n}\leq Aut(\Z_{2}^{n})$ is the permutation group on $n$ objects.

\item $\mathfrak{S}(\Z_{2}^{n},C_{n}).$\\
Let $C$ denote the cyclic permutation on the components $+$ and $-$ of $X$ in 
$\Z_{2}^{n}$ such that
\begin{equation}
C(X)=C\left( x_{0},x_{1},...,x_{n-2},x_{n-1}\right) =\left(x_{1},x_{2},x_{3},...,x_{0}\right),
\end{equation}
that is, $C(x_{i})=x_{(i+1)\Mod{n}}$. The permutation $C$ is a generator of cyclic group 
$C_{n}=\left\langle C\right\rangle$ of order $n$. Let 
$X_{C}=Orb_{C_{n}}X=\{C^{i}(X):C^{i}\in C_{n} \}$. Therefore, $C_{n}$ defines a partition in
equivalent class on $\Z_{2}^{n}$ which is an $S$-partition.

\item $\mathfrak{S}(\Z_{2}^{n},\Delta_{n}).$\\
Let $\delta_{a}\in S_{n}$ act on $X\in\Z_{2}^{n}$ by decimation, that is, 
$\delta_{a}(x_{i})=x_{ai\Mod{n}}$ for all $x_{i}$ in $X$, $(a,n)=1$ and let $\Delta_{n}$ denote 
the set of this $\delta_{a}$. The set $\Delta_{n}$ is a group of order $\phi(n)$ isomorphic to
$\Z_{n}^{*}$, the group the units of $\Z_{n}$, where $\phi$ is called the Euler totient function. 
Clearly $\mathfrak{S}(\Z_{2}^{n},\Delta_{n})$ is an $S$-partition of $\Z_{2}^{n}$.

\item $\mathfrak{S}(\Z_{2}^{n},H_{n}).$\\
We note by $RY$ the reversed sequence $RY = (y_{n-1},...,y_{1},y_{0})$ of $Y$ and let $H_{n}$ 
denote the permutation automorphic subgroup $H_{n}=\{1,R\}\leq S_{n}\leq Aut(Z_{2}^{n})$. 
Hence $H_{n}$ defines a partition on $\Z_{2}^{n}$ and $\mathfrak{S}(\Z_{2}^{n},H_{n})$ is a 
Schur ring.
\end{enumerate}

In [11] was shown that if there is a circulant Hadamard matrix, then must be contained in 
$\G_{4m^{2}}(2m^{2}-m)$. Let $Sym(\Z_{2}^{4n})$ denote the $S$-subgroup of all symmetric binary
sequences in $\Z_{2}^{4n}$. By the result of Brualdi, no circulant Hadamard matrix exists in
$Sym(\Z_{2}^{4n})$. As we want circulant matrices, then the search must be do in 
$\mathfrak{S}(\Z_{2}^{n},C_{n})$. Finally, in the following section the relationship
between of the Schur ring $\mathfrak{S}(\Z_{2}^{4n},\Delta_{4n})$ and the circulant Hadamard matrices
will be shown.

Let $\mathcal{A}$ be a alphabet and let $\mathcal{A}^{*}$ denote all finite words defined 
on $\mathcal{A}$. Any subsequence of consequtive letters of a word is a \textit{subword}. 
Given a word $w=s_{1}s_{2}\cdots s_{n}$, the number $n$ is called the length of $w$ and we denoted
this by $\vert w\vert$. A subset 
$\X\subseteq\mathcal{A}^{*}$ is a \textit{code} if it satisfies the following condiction:
For all $n,m\geq1$ and $x_{1},...,x_{n},y_{1},...,y_{m}\in\X$
\begin{equation}
x_{1}\cdots x_{n}=y_{1}\cdots y_{m}\ \Rightarrow\ n=m\ \textsl{and}\ x_{i}=y_{i}\ \textsl{for}\ i=1,2,...,n. 
\end{equation}
In [13] was shown the relationship between $S$-ring on $\Z_{2}^{n}$ and binary codes. In particular
were used codes for to construct $S$-subgroups over $\Z_{2}^{n}$. An important result obtained is
that $\X=\{X,CX,C^{2}X,...,C^{n-1}X\}$ is a base for all $\Z_{2}^{n}$, with $X=-+++\cdots+++$. 
Then we will use this base to construct words that could be circulant Hadamard matrices.
Let $Y_{G}$ denote the orbit of some $Y$ in $\Z_{2}^{n}$ under the action of $G$. If for a code 
$\X$ it is true that $\X=Y_{G}$ for some $Y$ in $\Z_{2}^{n}$, then we will say that $\X$ 
is a $G$-code.

This paper is organized as follows. In section 2 some results about $S$-subgroups of decimation
$\mathbb{I}_{n}(a)$ in $\mathfrak{S}(\Z_{2}^{n},\Delta_{n})$ are obtained. In particular will 
be shown that if $\left\langle a\right\rangle\leq\left\langle b\right\rangle$, then 
$\mathbb{I}_{n}(b)\leq\mathbb{I}_{n}(a)$. This theorem is the great importance because say us
that we must search circulant Hadamard matrices in $\mathbb{I}_{n}(x)$ only if $x$ has order a 
prime number module $n$. In the section 3 the conjecture is proved. First, it is shown that
if $H_{C}$ is a circulant Hadamard matrix, then its orbit under the action of 
$\Delta_{4n}$ has order 2, where $H_{C}$ is fixed o reversed by all decimation in $\Delta_{4n}$. 
Then we take $H$ in $\mathbb{I}_{4n}(x)$ with $x^{p}\equiv1\Mod{4n}$, $p$ a prime, and we 
suppose that $H$ either is fixed or is reversed under the action of $\delta_{2n+x}$. The idea is 
to show that $H$ seen as a word has length an even number in contradiction with the result obtain 
by Turyn.

\section{Schur ring $\mathbb{I}_{n}(a)$}

In this section, we will construct $\Delta_{n}$-codes for $S$-subgroups of 
$\mathfrak{S}(\Z_{2}^{n},\Delta_{n})$. We will use the commutation relation 
$C^{i}\delta_{a}=\delta_{a}C^{ia}$ for to prove all of results. 

Let 
$$(\mathsf{P}_{Y}(0),\mathsf{P}_{Y}(1),...,\mathsf{P}_{Y}(n-1))$$
denote the autocorrelation vector of $Y$ in $\Z_{2}^{n}$, where
$$\mathsf{P}_{Y}(k)=\sum_{i=0}^{n-1}y_{i}y_{i+k}$$
is the periodic autocorrelation function at shift $k$ of $Y$. 
Let $\mathfrak{A}(\Z_{2}^{n})$
denote the set of all autocorrelation vectors and let 
$\theta:\Z_{2}^{n}\rightarrow\mathfrak{A}(\Z_{2}^{n})$ be the map defined by 
$\theta(Y)=(\mathsf{P}_{Y}(0),\mathsf{P}_{Y}(1),\dots ,\mathsf{P}_{Y}(n-1))$.

The decimation group $\Delta_{n}$ do not alter the set of values which $\mathsf{P}_{X}(k)$ takes
on, but merely the order in which they appear, i.e., if $Y=\delta_{a}X$ then 
$\mathsf{P}_{Y}(k)=\mathsf{P}_{X}(ka)$. Therefore, we have the commutative diagram
\begin{equation}\label{diagram}
\xymatrix{
 \Z_{2}^{n} \ar[d]^{\theta} \ar[r]^{\delta_{r}} & \Z_{2}^{n} \ar[d]^{\theta}\\
   \mathfrak{A}(\Z_{2}^{n}) \ar[r]^{\delta_{r}} & \mathfrak{A}(\Z_{2}^{n}) 
}
\end{equation}
and $\theta \circ \delta_{r} = \delta_{r}\circ \theta.$

Let $Y\in\Z_{2}^{n}$ such that $\theta(Y)=(n,d,d,...,d)$. Such a binary sequence is known as 
binary sequence with $2$-levels autocorrelation value and are important by its applications on
telecommunication. We want to construct a $\Delta_{n}$-code for some $S$-subgroup $H$ of 
$\mathfrak{S}(\Z_{2}^{n},\Delta_{n})$ containing such $Y$. From (\ref{diagram}) is followed that 
$\theta(Y)=\delta_{a}\theta(Y)=\theta(\delta_{a}Y)$, for all $\delta_{a}\in\Delta_{n}$. Hence $Y$ 
and $\delta_{a}Y$ have the same autocorrelation vector. For $Y$ fulfilling $\delta_{a}Y=Y$ for some 
$\delta_{a}$ in $\Delta_{n}$ we have the following definition

\begin{definition}
Let $a$ be a unit in $\Z_{n}^{*}$. A word $Y$ in $\Z_{2}^{n}$ is $\delta_{a}$-\textbf{invariant}
if $\delta_{a}Y=Y$. Denote by $\mathbb{I}_{n}(a)$ the set of these $Y$.
\end{definition}

If $Y$ is in $\mathbb{I}_{n}(a)$, then $\delta_{r}Y$ is in $\mathbb{I}_{n}(a)$, too. Also
$\delta_{a}(YZ)=\delta_{a}Y\delta_{a}Z=YZ$ for all $Y,Z$ in $\mathbb{I}_{n}(a)$. Then
$\mathbb{I}_{n}(a)$ is an $S$-subgroup of $\mathfrak{S}(\Z_{2}^{n},\Delta_{n})$. Now, we shall see
that all factorization of words in $\mathbb{I}_{n}(a)$ is relationated with cyclotomic coset
of $a$ module $n$. First, we have the following definition

\begin{definition}
Let $a$ relative prime to $n$. The cyclotomic coset of $a$ module $n$ is defined by
\begin{equation*}
\mathsf{C}_{s}=\{s,sa,sa^{2},\cdots,sa^{t-1}\}.
\end{equation*}
where $sa^{t}\equiv s\Mod{n}$. A subset $\{s_{1},s_{2},\dots ,s_{r}\}$ of 
$\Z_{n}$ is called complete set of representatives of cyclotomic coset of $a$ modulo $n$ if 
$\mathsf{C}_{s_{1}}$,$\mathsf{C}_{s_{2}}$,..., $\mathsf{C}_{s_{r}}$ are distinct and are a 
partition of $\Z_{n}$. 
\end{definition}

Take $Y=C^{i_{1}}XC^{i_{2}}X\cdots C^{i_{r}}X$ in $\mathbb{I}_{n}(a)$ with $X=-++\cdots++$. 
We want $\delta_{a}Y=Y$. Then
\begin{eqnarray*}
\delta_{a}Y&=&\delta_{a}C^{i_{1}}X\delta_{a}C^{i_{2}}X\cdots\delta_{a}C^{i_{r}}X\\
&=&C^{i_{1}a^{-1}}\delta_{a}XC^{i_{2}a^{-1}}\delta_{a}X\cdots C^{i_{r}a^{-1}}\delta_{a}X\\
&=&C^{i_{1}a^{-1}}XC^{i_{2}a^{-1}}X\cdots C^{i_{r}a^{-1}}X
\end{eqnarray*}
since $\delta_{a}X=X$. As must be $\delta_{a}Y=Y$, then $i_{k}\equiv a^{-1}i_{j}\Mod{n}$ or 
$i_{j}\equiv ai_{k}\Mod{n}$ for
$1\leq k,j\leq r$. Let $\mathsf{C}_{s}X$ denote the word $C^{s}XC^{sa}X\cdots C^{sa^{t_{s}-1}}X$.
Then all $Y$ in $\mathbb{I}_{n}(a)$ has the form 
$Y=\mathsf{C}_{s_{1}}^{\epsilon_{1}}X\mathsf{C}_{s_{2}}^{\epsilon_{2}}X\cdots\mathsf{C}_{s_{r}}^{\epsilon_{r}}X$, with $\epsilon_{i}=0,1$ and ${s_{1},s_{2},...,s_{r}}$ is a complete set of 
representative of cyclotomic coset of $a$ module $n$. As $\mathbb{I}_{n}(a)$ is an $S$-subgroup in 
$\mathfrak{S}(\Z_{2}^{n},\Delta_{n})$, $\delta_{r}\mathsf{C}_{s_{i}}X=\mathsf{C}_{s_{j}}X$ and

\begin{equation}\label{alpha_invariante}
\mathcal{X}_{\mathbb{I}(a)}=\{X,\mathsf{C}_{s_{1}}X,\ \mathsf{C}_{s_{2}}X,...,\ \mathsf{C}_{s_{r}}X\}
\end{equation}

is a $\Delta_{n}$-code for $\mathbb{I}_{n}(a)$. Hence $\X_{\mathbb{I}(a)}^{*}$ 
has order $2^{r+1}$, where $r$ is the number of cyclotomic cosets of $a$ module $n$

The following theorem is fundamental for the proof of the circulant Hadamard conjecture. This
show us the relationship between $S$-subgroups $\mathbb{I}_{n}(a)$ and $\mathbb{I}_{n}(b)$
when $\left\langle a\right\rangle$ and $\left\langle b\right\rangle$ are subgroups from the other.
Firstly, we obtain the following

\begin{lemma}\label{lemma_action_decimation}
Take a codeword $\mathsf{C}_{s}X$ in $\X_{\mathbb{I}_{n}(a)}$. 
Then $\delta_{r}\mathsf{C}_{s}X=\mathsf{C}_{sr^{-1}}X$ for all $\delta_{r}$ in $\Delta_{n}$.
\end{lemma}
\begin{proof}
We have
\begin{eqnarray*}
\delta_{r}\mathsf{C}_{s}X&=&\delta_{r}(C^{s}XC^{sa}X\cdots C^{sa^{t_{s}-1}}X)\\
&=&\delta_{r}C^{s}X\delta_{r}C^{sa}X\cdots\delta_{r}C^{sa^{t_{s}-1}}X\\
&=&C^{sr^{-1}}\delta_{r}XC^{sar^{-1}}\delta_{r}X\cdots C^{sa^{t_{s}-1}r^{-1}}\delta_{r}X\\
&=&C^{sr^{-1}}XC^{sar^{-1}}X\cdots C^{sa^{t_{s}-1}r^{-1}}X\\
&=&\mathsf{C}_{sr^{-1}}X.
\end{eqnarray*}
\end{proof}

\begin{theorem}\label{theo_principal}
If $\left\langle b\right\rangle$ is a subgroup of $\left\langle a\right\rangle$, then 
$\mathbb{I}_{n}(a)\subseteq\mathbb{I}_{n}(b)$.
\end{theorem}
\begin{proof}
Let $\mathsf{C}_{1}^{a}$ and $\mathsf{C}_{1}^{b}$ denote the classes $\{1,a,a^{2},...,a^{t-1}\}$
and $\{1,b,b^{2},...,b^{s-1}\}$. By hypothesis
$\left\langle b\right\rangle\leq\left\langle a\right\rangle$, then $\mathsf{C}_{1}^{b}\subseteq\mathsf{C}_{1}^{a}$. Hence there exists $y_{i}$ in $\left\langle a\right\rangle$ such that
$$\mathsf{C}_{1}^{a}=\mathsf{C}_{1}^{b}\cup y_{1}\mathsf{C}_{1}^{b}\cup\cdots\cup y_{k}\mathsf{C}_{1}^{b},$$
and $k=[\left\langle a\right\rangle:\left\langle b\right\rangle]$. As clearly 
$y_{i}\mathsf{C}_{1}^{b}=\mathsf{C}_{y_{i}}^{b}$, then
\begin{equation}
\mathsf{C}_{1}^{a}X=\mathsf{C}_{1}^{b}X\mathsf{C}_{y_{1}}^{b}X\cdots\mathsf{C}_{y_{k}}^{b}X
\end{equation}
and $\mathsf{C}_{1}^{a}X\in\mathbb{I}_{n}(b)$. Now, take $\delta_{r}$ in $\Delta_{n}$.
Then 
\begin{eqnarray*}
\delta_{r}\mathsf{C}_{1}^{a}X&=&\delta_{r}(\mathsf{C}_{1}^{b}X\mathsf{C}_{y_{1}}^{b}X\cdots\mathsf{C}_{y_{k}}^{b}X)\\
&=&\delta_{r}\mathsf{C}_{1}^{b}X\delta_{r}\mathsf{C}_{y_{1}}^{b}X\cdots\delta_{r}\mathsf{C}_{y_{k}}^{b}X\\
&=&\mathsf{C}_{r^{-1}}^{b}X\mathsf{C}_{y_{1}r^{-1}}^{b}X\cdots\mathsf{C}_{y_{k}r^{-1}}^{b}X.
\end{eqnarray*}
From Lemma \ref{lemma_action_decimation} is followed that $\mathsf{C}_{r^{-1}}^{a}X$ is in 
$\mathbb{I}_{n}(b)$ for all $r\in\Z_{n}^{*}$. Accordingly 
$\mathbb{I}_{n}(a)\subseteq\mathbb{I}_{n}(b)$.
\end{proof}

We will use the above theorem to obtain the following corollaries

\begin{corollary}
Let $a$ be an unity in $\Z_{4n}^{*}$ of order $2k$. Then $\mathbb{I}_{4n}(a)\subset\mathbb{I}_{4n}(a^{k})$.
\end{corollary}
\begin{proof}
It is enough with to note that $\{1,a^{k}\}$ is a subgroup of $\left\langle a\right\rangle$.
\end{proof}

\begin{corollary}
Let
\begin{equation*}
\{1\}\leq\left\langle a_{1}\right\rangle\leq\left\langle a_{2}\right\rangle\leq\cdots\leq\left\langle b\right\rangle
\end{equation*}
a serie of cyclic subgroups of $\left\langle b\right\rangle$. Then
\begin{equation}
\mathbb{I}_{n}(b)\subset\cdots\subset\mathbb{I}_{n}(a_{2})\subset\mathbb{I}_{n}(a_{1})\subset\mathbb{I}_{n}(1)=\Z_{2}^{n}.
\end{equation}
\end{corollary}
\begin{proof}
A consequence trivial of Theorem \ref{theo_principal}.
\end{proof}


Finally, we will use the previous results to show some lattice of $S$-subgroups $\mathbb{I}_{4n}(a)$
in $\Z_{2}^{4n}$ ordered by inclusion.

\subsection{Some Example}

$S$-subgroups $\mathbb{I}_{4n}(a)$ in $\Z_{2}^{4n}$ was obtained by using \textsc{sage}. It was taken
into account that $\mathbb{I}_{4n}(a)=\mathbb{I}_{4n}(a^{-1})$. Lattice of $S$-subgroups 
$\mathbb{I}_{4n}(a)$ in $\Z_{2}^{4n}$ ordered by inclusion will be shown for $4n=24,36,60,196,668$.

\subsubsection{4n=24}

The cyclic subgroups of $\Z_{24}^{*}$ are $\left\langle5\right\rangle$,$\left\langle7\right\rangle$
,$\left\langle11\right\rangle$, $\left\langle13\right\rangle$, $\left\langle17\right\rangle$,  
$\left\langle19\right\rangle$ and $\left\langle23\right\rangle$, all of order 2. Then the lattice of
$S$-subgroups $\mathbb{I}_{24}(a)$ in $\Z_{2}^{24}$ is

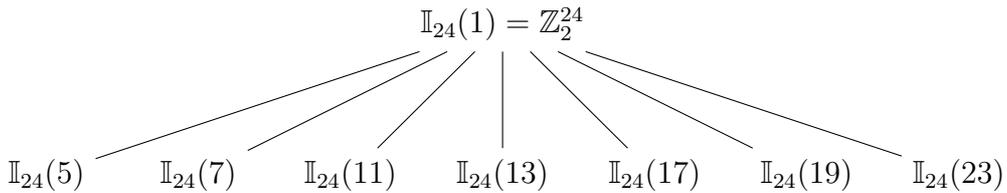
\begingroup\makeatletter\def\f@size{12}\check@mathfonts
\begin{center}
\begin{tikzpicture}
  \node (a) at (1,3) {$\mathbb{I}_{24}(5)$};
  \node (b) at (3,3) {$\mathbb{I}_{24}(7)$};
  \node (c) at (5,3) {$\mathbb{I}_{24}(11)$};
  \node (d) at (7,3) {$\mathbb{I}_{24}(13)$};
  \node (e) at (9,3) {$\mathbb{I}_{24}(17)$};
  \node (f) at (11,3) {$\mathbb{I}_{24}(19)$};
  \node (g) at (13,3) {$\mathbb{I}_{24}(23)$};
  \node (max) at (7,5) {$\mathbb{I}_{24}(1)=\Z_{2}^{24}$};
  \draw
  (a) -- (max)
  (b) -- (max)
  (c) -- (max)
  (d) -- (max)
  (e) -- (max)
  (f) -- (max)
  (g) -- (max);
  \end{tikzpicture}
\end{center}
\endgroup

\subsubsection{4n=36}
The cyclic subgroups of $\Z_{36}^{*}$ are $\left\langle5\right\rangle$, $\left\langle7\right\rangle$,
$\left\langle11\right\rangle$, $\left\langle13\right\rangle$, $\left\langle17\right\rangle$,
$\left\langle19\right\rangle$ and $\left\langle35\right\rangle$. Then the lattice of
$S$-subgroups $\mathbb{I}_{36}(a)$ in $\Z_{2}^{36}$ is

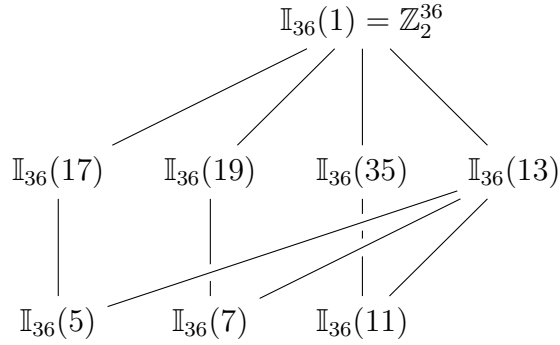
\begingroup\makeatletter\def\f@size{12}\check@mathfonts
\begin{center}
\begin{tikzpicture}
  \node (a) at (3,3) {$\mathbb{I}_{36}(5)$};
  \node (b) at (3,5) {$\mathbb{I}_{36}(17)$};
  \node (c) at (5,3) {$\mathbb{I}_{36}(7)$};
  \node (d) at (5,5) {$\mathbb{I}_{36}(19)$};
  \node (e) at (7,3) {$\mathbb{I}_{36}(11)$};
  \node (f) at (7,5) {$\mathbb{I}_{36}(35)$};
  \node (g) at (9,5) {$\mathbb{I}_{36}(13)$};
  \node (max) at (7,7) {$\mathbb{I}_{36}(1)=\Z_{2}^{36}$};
  \draw
  (a) -- (b) -- (max)
  (c) -- (d) -- (max)
  (e) -- (f) -- (max)
  (e) -- (g)
  (g) -- (max);
  \draw[preaction={draw=white, -,line width=6pt}]
  (a) -- (g)  
  (c) -- (g);
  \end{tikzpicture}
\end{center}
\endgroup

The subgroups $\left\langle17\right\rangle$, $\left\langle19\right\rangle$ and 
$\left\langle35\right\rangle$ have order 2 and the subgroup $\left\langle13\right\rangle$ has
order 3.

\subsubsection{4n=60}
The cyclic subgroups of $\Z_{60}^{*}$ are $\left\langle7\right\rangle$, 
$\left\langle11\right\rangle$, $\left\langle13\right\rangle$, $\left\langle17\right\rangle$, 
$\left\langle19\right\rangle$, $\left\langle23\right\rangle$, $\left\langle29\right\rangle$,
$\left\langle31\right\rangle$, $\left\langle41\right\rangle$, $\left\langle49\right\rangle$ and
$\left\langle59\right\rangle$. Then the lattice of $S$-subgroups $\mathbb{I}_{60}(a)$ in 
$\Z_{2}^{60}$ is

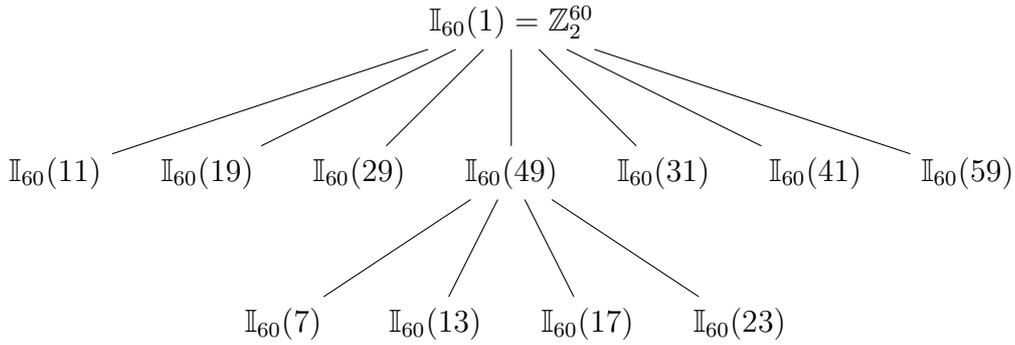
\begingroup\makeatletter\def\f@size{12}\check@mathfonts
\begin{center}
\begin{tikzpicture}
  \node (b) at (3,3) {$\mathbb{I}_{60}(11)$};
  \node (c) at (5,3) {$\mathbb{I}_{60}(19)$};
  \node (d) at (7,3) {$\mathbb{I}_{60}(29)$};
  \node (a) at (9,3) {$\mathbb{I}_{60}(49)$};  
  \node (a1) at (6,1) {$\mathbb{I}_{60}(7)$};  
  \node (a2) at (8,1) {$\mathbb{I}_{60}(13)$};  
  \node (a3) at (10,1) {$\mathbb{I}_{60}(17)$};  
  \node (a4) at (12,1) {$\mathbb{I}_{60}(23)$};  
  \node (e) at (11,3) {$\mathbb{I}_{60}(31)$};
  \node (f) at (13,3) {$\mathbb{I}_{60}(41)$};
  \node (g) at (15,3) {$\mathbb{I}_{60}(59)$};
  \node (max) at (9,5) {$\mathbb{I}_{60}(1)=\Z_{2}^{60}$};
  \draw
  (a) -- (max)
  (a1) -- (a)
  (a2) -- (a)
  (a3) -- (a)
  (a4) -- (a)
  (b) -- (max)
  (c) -- (max)
  (d) -- (max)
  (e) -- (max)
  (f) -- (max)
  (g) -- (max);
  \end{tikzpicture}
\end{center}
\endgroup

where the subgroups $\left\langle11\right\rangle$, $\left\langle19\right\rangle$,
$\left\langle29\right\rangle$, $\left\langle49\right\rangle$, $\left\langle31\right\rangle$,
$\left\langle41\right\rangle$ and $\left\langle59\right\rangle$ have order 2.

\subsubsection{4n=196}
The cyclic subgroups of $\Z_{196}^{*}$ are $\left\langle3\right\rangle$, 
$\left\langle5\right\rangle$, $\left\langle9\right\rangle$, $\left\langle11\right\rangle$, 
$\left\langle13\right\rangle$, $\left\langle15\right\rangle$, $\left\langle19\right\rangle$,
$\left\langle27\right\rangle$, $\left\langle29\right\rangle$, $\left\langle67\right\rangle$,
$\left\langle97\right\rangle$, $\left\langle99\right\rangle$, $\left\langle117\right\rangle$,
$\left\langle165\right\rangle$ and $\left\langle195\right\rangle$. Then the lattice of 
$S$-subgroups $\mathbb{I}_{196}(a)$ in $\Z_{2}^{196}$ is

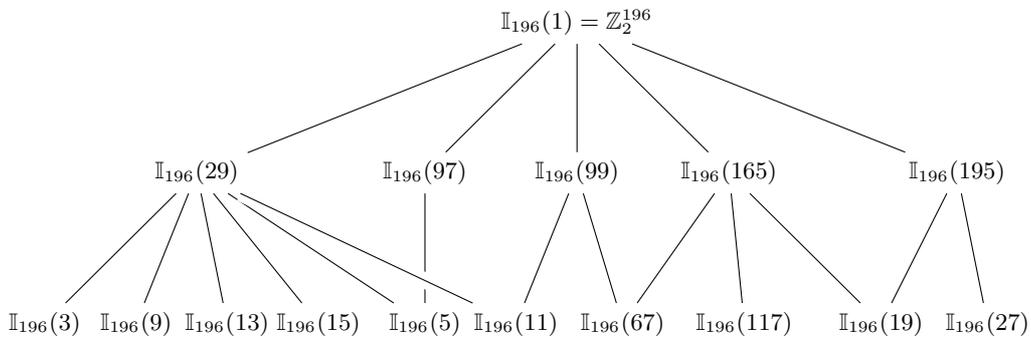
\begingroup\makeatletter\def\f@size{9}\check@mathfonts
\begin{center}
\begin{tikzpicture}
  \node (b) at (6,3) {$\mathbb{I}_{196}(29)$};
  \node (b1) at (4,1) {$\mathbb{I}_{196}(3)$};  
  \node (b2) at (5.2,1) {$\mathbb{I}_{196}(9)$};  
  \node (b3) at (6.4,1) {$\mathbb{I}_{196}(13)$};  
  \node (b4) at (7.6,1) {$\mathbb{I}_{196}(15)$};    
  \node (c) at (9,3) {$\mathbb{I}_{196}(97)$};
  \node (c1) at (9,1) {$\mathbb{I}_{196}(5)$};
  \node (d) at (11,3) {$\mathbb{I}_{196}(99)$};
  \node (d1) at (10.2,1) {$\mathbb{I}_{196}(11)$};
  \node (d2) at (11.6,1) {$\mathbb{I}_{196}(67)$};
  \node (a) at (13,3) {$\mathbb{I}_{196}(165)$};  
  \node (a1) at (13.2,1) {$\mathbb{I}_{196}(117)$};  
  \node (e) at (16,3) {$\mathbb{I}_{196}(195)$};  
  \node (e1) at (15,1) {$\mathbb{I}_{196}(19)$};  
  \node (e2) at (16.4,1) {$\mathbb{I}_{196}(27)$};  
  \node (max) at (11,5) {$\mathbb{I}_{196}(1)=\Z_{2}^{196}$};
  \draw
  (a) -- (max)
  (a1) -- (a)
  (b1) -- (b)
  (b2) -- (b)
  (b3) -- (b)
  (b4) -- (b)
  (b) -- (max)
  (c) -- (max)
  (c1) -- (c)
  (c1) -- (b)
  (d) -- (max)
  (d1) -- (d)
  (d1) -- (b)
  (d2) -- (d)
  (d2) -- (a)
  (e) -- (max)
  (e1) -- (e)
  (e1) -- (a)
  (e2) -- (e);
  \draw[preaction={draw=white, -,line width=6pt}]
  (d1) -- (b);
  \end{tikzpicture}
\end{center}
\endgroup

where the subgroups $\left\langle97\right\rangle$, $\left\langle99\right\rangle$ and
$\left\langle195\right\rangle$ have order 2 and the subgroups $\left\langle29\right\rangle$ and
$\left\langle165\right\rangle$ have order 7 and 3, respectively.

\subsubsection{4n=668}
The cyclic subgroups of $\Z_{668}^{*}$ are $\left\langle3\right\rangle$, 
$\left\langle5\right\rangle$, $\left\langle9\right\rangle$, $\left\langle15\right\rangle$,
$\left\langle333\right\rangle$, $\left\langle335\right\rangle$, 
and $\left\langle667\right\rangle$. Then the lattice of $S$-subgroups $\mathbb{I}_{668}(a)$ in 
$\Z_{2}^{668}$ is

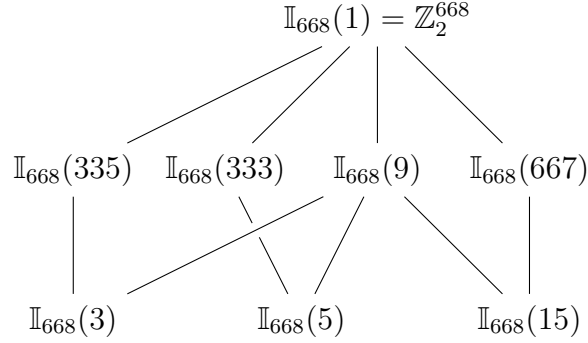
\begingroup\makeatletter\def\f@size{12}\check@mathfonts
\begin{center}
\begin{tikzpicture}
  \node (b) at (7,3) {$\mathbb{I}_{668}(9)$};
  \node (c) at (5,3) {$\mathbb{I}_{668}(333)$};
  \node (c1) at (6,1) {$\mathbb{I}_{668}(5)$};  
  \node (d) at (3,3) {$\mathbb{I}_{668}(335)$};
  \node (d1) at (3,1) {$\mathbb{I}_{668}(3)$};    
  \node (a) at (9,3) {$\mathbb{I}_{668}(667)$};  
  \node (a1) at (9,1) {$\mathbb{I}_{668}(15)$};  
  \node (max) at (7,5) {$\mathbb{I}_{668}(1)=\Z_{2}^{668}$};
  \draw
  (a) -- (max)
  (a1) -- (a)
  (b) -- (max)
  (c) -- (max)
  (c1) -- (c)
  (d) -- (max)
  (d1) -- (d);
 \draw[preaction={draw=white, -,line width=6pt}]
  (c1) -- (b)
  (d1) -- (b)
  (a1) -- (b);
  \end{tikzpicture}
\end{center}
\endgroup

where the subgroups $\left\langle333\right\rangle$, $\left\langle335\right\rangle$ and
$\left\langle667\right\rangle$ have order 2 and the subgroup
$\left\langle9\right\rangle$ has order 83.

As noted, it is sufficient to search for circulant Hadamard matrices in those $S$-subgroups
$\mathbb{I}_{4n}(x)$ such that $x$ has order a prime number.

\section{There is no circulant Hadamard matrices in $\Z_{2}^{4n}$}

In this section the conjecture is proved. Firstly, we show that the orbit of all circulant 
Hadamard matrix $H_{C}$ has order 2 under the action of $\Delta_{4n}$, where $H_{C}$ is fixed o
reversed by all decimation in $\Delta_{4n}$. Then we take $H$ in $\mathbb{I}_{4n}(x)$ with 
$x^{p}\equiv1\Mod{4n}$, $p$ a prime, and we suppose that $H$ either is fixed or is reversed under 
the action of $\delta_{2n+x}$. The idea is to show that $H$ seen as a word has length an even 
number in contradiction with the result obtain by Turyn. All word in $\mathbb{I}_{4n}(x)$ is
factorized in terms of some set of subwords here constructed. A very useful decimation in our 
proof will be $\delta_{2n+1}$, since $\delta_{2n+1}\in\Delta_{4n}$ for all $n\geq1$. We start with 
the construction of its associated $S$-subgroup.

\begin{proposition}\label{prop_inv_2n+1}
The $S$-subgroup $\mathbb{I}_{4n}(2n+1)$ has order $2^{3n}$.
\end{proposition}
\begin{proof}
As $(2n+1)^{2}\equiv1\mod4n$, then $\{1,2n+1\}$ is a cyclotomic coset of $2n+1$ module $4n$,
and
\begin{eqnarray*}
\mathsf{C}_{0}&=&\{0\},\\
\mathsf{C}_{2a}&=&\{2a\},\text{ $0\leq a\leq 2n-1$},\\
\mathsf{C}_{2b+1}&=&\{2b+1,2n+2b+1\},\text{ $0\leq b\leq n-1$}
\end{eqnarray*}
are all of cyclotomic cosets of $2n+1$ module $4n$. Hence 
\begin{equation}
\X_{\mathbb{I}_{4n}(2n+1)}=\{\mathsf{C}_{0}X,\mathsf{C}_{2a}X,\mathsf{C}_{2b+1}X\}
\end{equation}
and $\vert\mathbb{I}_{4n}(2n+1)\vert=2^{3n}$ as we desirable.
\end{proof}

A first useful result to show that the orbit of a circulant Hadamard matrix has order 2 is 
the following. Here will be shown the relationship between the decimations $\delta_{x}$ and 
$\delta_{n-x}$ and $Y$ when this is either fixed or reversed for those.

\begin{theorem}\label{theo_relation_RY_Y}
If $\delta_{x}Y_{C}=RY_{C}$ for some $x\in\Z_{n}^{*}$, $x\neq1,n-1$, then $Y\in\mathbb{I}_{n}(n-x)$.
\end{theorem}
\begin{proof}
Obviously the hypothesis $\delta_{1}Y_{C}=RY_{C}$ is never holded. On the other hand, if 
$\delta_{2n-1}Y_{C}=RY_{C}$, then $Y\notin\mathbb{I}_{2n}(2n-1)$. In this case 
$Y\in\mathbb{I}_{2n}(1)=\Z_{2}^{2n}$ trivially. The condiction $\delta_{2n}Y_{C}=RY_{C}$ is holds
even when $Y$ is contained in $\mathbb{I}_{2n+1}(2n)$ since 
$\mathbb{I}_{2n+1}(2n)=Sym(\Z_{2}^{2n+1})$ (see [13], Theorem 14). Then 
$Y\in\mathbb{I}_{2n+1}(1)$ trivially. Now suppose $x\neq1,n-1$ with 
$Y=C^{i_{1}}XC^{i_{2}}X\cdots C^{i_{r}}X$. As we want $\delta_{x}Y_{C}=RY_{C}$,
then
\begin{eqnarray*}
\delta_{x}Y_{C}&=&(C^{i_{1}x^{-1}}XC^{i_{2}x^{-1}}X\cdots C^{i_{r}x^{-1}}X)_{C}
\end{eqnarray*}
and
\begin{eqnarray*}
RY_{C}&=&(C^{n-i_{1}}XC^{n-i_{2}}X\cdots C^{n-i_{r}}X)_{C}
\end{eqnarray*}
implies that $i_{k}x^{-1}\equiv n-i_{l}\Mod{n}$ or $i_{k}\equiv n-i_{l}x\equiv(n-x)i_{l}\Mod{n}$.
Hence $Y\in\mathbb{I}_{n}(n-x)$.
\end{proof}

Brualdi[3] proved that if there is a circulant Hadamard matrix, then this is nonsymmetric. We 
will use this for to show that a circulant Hadamard matrix never is fixed by all decimation in 
$\Delta_{4n}$.

\begin{theorem}
Suppose $H$ in $\bigcap_{x\in\Z_{4n}^{*}}\mathbb{I}_{4n}(x)$, then $H_{C}$ is no Hadamard.
\end{theorem}
\begin{proof}
Suppose that $H_{C}$ is Hadamard. Then $\theta(H)=\theta(RH)=\theta(\delta_{x}H)$ for all
$x\in\Z_{4n}^{*}$. From the above theorem there is $y\in\Z_{4n}^{*}$ such that 
$\delta_{y}H=RH$. Therefore $RH=H$ and $H$ is symmetric, but there is no symmetric circulant 
Hadamard matrices. Consequently, if $H_{C}$ is Hadamard then never is contained in 
$\bigcap_{x\in\Z_{4n}^{*}}\mathbb{I}_{4n}(x)$.
\end{proof}

As was already announced we will show that the orbit of a circulant Hadamard matrix has order 2.
We take in mind the three commutation relations among $\delta_{r}$, $C$ and $R$:
\begin{eqnarray}
\delta_{r}R&=&R\delta_{r}C^{r-1},\\
RC&=&C^{n-1}R,\\
C^{i}\delta_{r}&=&\delta_{r}C^{ir}
\end{eqnarray}

\begin{theorem}\label{theo_size_orbit}
Let $H_{C}$ be a circulant Hadamard matrix of order $4n$. Then the orbit of $H_{C}$ has order 
$2$ under the action of $\Delta_{4n}$.
\end{theorem}
\begin{proof}
It was already proven that no circulant Hadamard matrix is fixed under the action of 
$\Delta_{4n}$. Hence the orbit of $H_{C}$ under the action of $\Delta_{4n}$ is $\geq2$. Suppose
$H_{C}$ in $\mathbb{I}_{4nC}(x)$. As $\delta_{x}RH_{C}=R\delta_{x}H_{C}=RH_{C}$, then is followed 
that $RH_{C}\in\mathbb{I}_{4nC}(x)$. Now, suppose $p\in\Z_{4n}^{*}$ such that 
$\delta_{p}H_{C}=RH_{C}$. Then 
\begin{equation*}
\delta_{p}RH_{C}=R\delta_{p}C^{p-1}H_{C}=R\delta_{p}H_{C}=R^{2}H_{C}=H_{C}.
\end{equation*}
Thus $\delta_{p}H_{C}=RH_{C}$ implies that $\delta_{p^{-1}}H_{C}=RH_{C}$. In fact, $\delta_{p}$
and $\delta_{p^{-1}}$ belong to $\delta_{4n-1}S(H_{C})$, where $S(H_{C})$ is the stabilizer of 
$H_{C}$ in $\Delta_{4n}$. Let $\Delta_{4n}/S(H_{C})$ denote the set of cosets 
$\{\delta_{x}S(H_{C})\}$ of $S(H_{C})$ in $\Delta_{4n}$. We want to show that 
$[\Delta_{4n}:S(H_{C})]=2$. As $(\delta_{a}Y)C^{i}(\delta_{a}Y)=\delta_{a}(YC^{ia}Y)$, then 
$(\delta_{a}Y)C^{i}(\delta_{a}Y)$ is a decimation of $YC^{ia}Y$. From the commutation relations 
above $\delta_{x}H=H$ implies $\delta_{4n-x}H=C^{4n-1}RH$ for $\delta_{x}\in S(H_{C})$. Hence 
$\delta_{x}(HC^{ix}H)=HC^{i}H$ and $\delta_{4n-x}(HC^{4n-ix}H)=RC(HC^{4n-i}H)$ implies 
$\delta_{x}(HC^{ix}H)_{C}=(HC^{i}H)_{C}$ and 
$\delta_{4n-x}(HC^{4n-ix}H)_{C}=R(HC^{4n-i}H)_{C}$, respectively, for $\delta_{x}\in S(H_{C})$. 
Let $Y_{C}^{2}$ denote the vector
\begin{equation}
Y_{C}^{2}=(\textbf{1},(YCY)_{C},(YC^{2}Y)_{C},...,(YC^{n-1}Y)_{C})
\end{equation}
for any $Y\in\Z_{2}^{n}$.
As $(YC^{i}Y)_{C}=(YC^{n-i}Y)_{C}$, then $Y_{C}^{2}$ is in correspondence with the autocorrelation
vector 
$$\theta(Y)=(\mathsf{P}_{Y}(0),\mathsf{P}_{Y}(1),...,\mathsf{P}_{Y}(n-1)),$$
where $\mathsf{P}_{Y}(k)=n-2\vert YC^{k}Y\vert$. Then 
$\delta_{a}\theta(H)=\theta(\delta_{a}H)=\theta(H)=\theta(RH)$,
for $H_{C}$ being Hadamard and all $\delta_{a}$ in $\Delta_{4n}$, implies that either
\begin{eqnarray*}
\delta_{a}H_{C}^{2}&=&(\delta_{a}\textbf{1},\delta_{a}(HCH)_{C},\delta_{a}(HC^{2}H)_{C},...,\delta_{a}(HC^{4n-1}H)_{C})\\
&=&(\textbf{1},(HC^{a^{-1}}H)_{C},(HC^{2a^{-1}}H)_{C},...,(HC^{4n-a^{-1}}H)_{C})\\
&=&(\textbf{1},(HCH)_{C},(HC^{2}H)_{C},...,(HC^{4n-1}H)_{C})\\
&=&H_{C}^{2}
\end{eqnarray*}
or $\delta_{a}H_{C}^{2}=RH_{C}^{2}$. Thus either $\delta_{a}H=H$ or $\delta_{a}H=C^{4n-1}RH$. 
Hence $\Delta_{4n}/S(H_{C})$ has only two cosets, since $\delta_{a}$ 
is arbitrary. Hence the orbit of $H_{C}$ under $\Delta_{4n}$ is $\{H_{C},RH_{C}\}$.
\end{proof}

\begin{remark}
In the proof above neither the length nor orthogonality of $H_{C}$ were considered, but 
the autocorrelation vector $(4n,0,...,0)$. By changing that vector for $(n,d,...,d)$ we will 
obtain the following result: Let $Y$ be a binary sequence with 2-level autocorrelation values. Then
the orbit of $Y_{C}$ has order $\leq2$ under the action of $\Delta_{n}$.
\end{remark}

On the other hand, it was already established that $\mathbb{I}_{n}(a)$ is an $S$-subgroup in
$\mathfrak{S}(\Z_{2}^{n},\Delta_{n})$. Now we will show that $C^{n}Y$ is in $\mathbb{I}_{2n}(a)$ 
for all $Y$ in $\mathbb{I}_{2n}(a)$. Thus the action of $C^{n}$ determines a partition on 
$\mathbb{I}_{2n}(a)$. This it is important because we can obtain the structure of all 
word $Y$ in $\mathbb{I}_{4n}(a)$ such that $YC^{2n}Y$ has length $2n$. If we can prove that such 
$Y$ is not Hadamard, then we will have reached the proof of the conjecture.

\begin{lemma}
$C^{n}\mathsf{C}_{s}X=\mathsf{C}_{s+n}X$ with $\mathsf{C}_{s}X$ a codeword in 
$\X_{\mathbb{I}_{2n}(a)}$ for all $a\in\Z_{2n}^{*}$.
\end{lemma}
\begin{proof}
This is easily seen by noting that all $x$ in $\Z_{2n}^{*}$ is an odd number and from
$nx\equiv n\Mod{2n}$. Hence
\begin{eqnarray*}
C^{n}\mathsf{C}_{s}X&=&C^{s+n}XC^{sa+n}X\cdots C^{sa^{t_{s}-1}+n}X\\
&=&C^{s+n}XC^{(s+n)a}X\cdots C^{(s+n)a^{t_{s}-1}}X\\
&=&\mathsf{C}_{s+n}X.
\end{eqnarray*}
\end{proof}

Now we have the following

\begin{definition}
We will call to $\mathsf{C}_{s}X$ and $\mathsf{C}_{s+2n}X$ codewords
$C^{2n}$-complementary in $\X_{\mathbb{I}_{4n}(a)}$. If $\mathsf{C}_{s}X=\mathsf{C}_{s+2n}X$, 
then $\mathsf{C}_{s}X$ is a codeword $C^{2n}$-invariant.
\end{definition}

In the following lemmas we will write all word in
$\mathbb{I}_{4n}(x)$, $x^{p}\equiv1\Mod{4n}$, $p$ prime in term of 
$C^{2n}$-complementary and $C^{2n}$-invariant subwords.

\begin{lemma}\label{lemma_2n_factor}
Take $x$ in $\Z_{4n}^{*}$ such that $x^{2}\equiv1\Mod{4n}$. Then each word in $\mathbb{I}_{4n}(x)$
can contain the following subwords 
\begin{eqnarray*}
A_{r}X&=&C^{a_{1}}X\cdots C^{a_{r}}X,\text{ where $a_{i}-a_{i+1}\neq 2n$ and $a_{i}x\equiv a_{i}$}\\
B_{s}X&=&C^{b_{1}}XC^{b_{1}+2n}X\cdots C^{b_{s}}XC^{b_{s}+2n}X,\text{ where $b_{i}-b_{i+1}\not\equiv2n$ and $b_{i}x\equiv b_{i}$}\\
D_{t}X&=&C^{d_{1}}XC^{d_{1}x}X\cdots C^{d_{t}}XC^{d_{t}x}X,\text{ where $d_{i}-d_{i+1}\not\equiv2n$}\\
E_{u}X&=&C^{e_{1}}XC^{e_{1}x}XC^{e_{1}+2n}XC^{(e_{1}+2n)x}X,\\
&&\cdots C^{e_{u}}XC^{e_{u}x}XC^{e_{u}+2n}XC^{(e_{u}+2n)x}X,\text{ where $e_{i}-e_{i+1}\not\equiv2n$}\\
F_{v}X&=&C^{f_{1}}XC^{f_{1}x}X\cdots C^{f_{v}}XC^{f_{v}x}X,\text{ where $f_{i}-f_{i+1}\not\equiv2n$}
\end{eqnarray*}
such that $A_{r}X$ and $D_{t}X$ contain no any $C^{2n}$-complementary pairs and
$C^{2n}(B_{s}X)=B_{s}X$, $C^{2n}(E_{u}X)=E_{u}X$ and $C^{2n}(F_{v}X)=F_{v}X$.
\end{lemma}
\begin{proof}
All cyclotomic cosets of $x$ have order either 1 or 2. Thus the subwords $A_{r}X$ and $B_{s}X$ are 
formed by cosets of order 1 and the subwords $D_{t}X$, $E_{u}X$ and $F_{v}X$ are formed by 
cosets of order 2. The properties of the subwords follow of its definition.
\end{proof}

\begin{lemma}\label{lemma_2n_subword_prime}
Let $x^{p}\equiv1\mod4n$, $p$ prime. All word in $\mathbb{I}_{4n}(x)$ contains some of the 
following subwords: 
\begin{eqnarray*}
A_{r}X&=&C^{a_{1}}X\cdots C^{a_{t}}X, \text{ where $a_{i}x\equiv a_{i}$},\\
B_{s}X&=&C^{b_{1}}XC^{b_{1}+2n}X\cdots C^{b_{t}}XC^{b_{t}+2n}X \text{ where $b_{i}x\equiv b_{i}$},\\
D_{t}X&=&\mathsf{C}_{d_{1}}X\cdots\mathsf{C}_{d_{r}}X,\text{ where $\vert\mathsf{C}_{d_{i}}X\vert=p$},\\
E_{u}X&=&\mathsf{C}_{e_{1}}X\mathsf{C}_{e_{1}+2n}X\cdots\mathsf{C}_{e_{s}}X\mathsf{C}_{e_{s}+2n}X,\text{ where $\vert\mathsf{C}_{e_{i}}X\vert=p$}.
\end{eqnarray*}
\end{lemma}
\begin{proof}
As $x^{i}2n\equiv2n\mod4n$, then no codeword in $\X_{\mathbb{I}_{4n}(x)}$ is $C^{2n}$-invariant.
Hence $A_{r}X,B_{s}X,D_{t}X$ and $E_{u}X$ are the only subwords forming all word in 
$\mathbb{I}_{4n}(x)$.
\end{proof}

The following lemma allows us to note the decimation $\delta_{p}$ preserves the structure of 
any written word with the subwords $A_{r}X$, $B_{s}X$, $D_{t}X$, $E_{u}X$ and $F_{v}X$. Thus
we can to find the structure of the subwords and be able to construct the desired words 

\begin{lemma}\label{lemma_conservated_struct}
The action of $\delta_{p}$ conserves the structure of $A_{r}X$, $B_{s}X$, $D_{t}X$, $E_{u}X$
and $F_{v}X$ in $\mathbb{I}_{4n}(x)$, $x^{q}\equiv1\Mod{4n}$, $q$ a number prime.
\end{lemma}
\begin{proof}
Suppose $x^{2}\equiv1\Mod{4n}$. We want to show that $\delta_{p}A_{r}X$ contains no 
$C^{2n}$-complementary pairs. We have
\begin{eqnarray*}
C^{2n}\delta_{p}A_{r}X&=&C^{2n}\delta_{p}(C^{a_{1}}XC^{a_{2}}X\cdots C^{a_{r}}X)\\
&=&C^{2n}\delta_{p}C^{a_{1}}XC^{2n}\delta_{p}C^{a_{2}}X\cdots C^{2n}\delta_{p}C^{a_{r}}X\\
&=&C^{2n}C^{a_{1}p^{-1}}XC^{2n}C^{a_{2}p^{-1}}X\cdots C^{2n}C^{a_{r}p^{-1}}X\\
&=&C^{2n+a_{1}p^{-1}}XC^{2n+a_{2}p^{-1}}X\cdots C^{2n+a_{r}p^{-1}}X.
\end{eqnarray*}
If 
$$2n+a_{i}p^{-1}\equiv a_{j}p^{-1}\Mod{4n},$$
then $p^{-1}(a_{i}-a_{j})\equiv2n\Mod{4n}$. But $p^{-1}\not\vert4n$ and the above is not
possible for any pair $(i,j)$. Hence $\delta_{p}A_{r}X$ contains no $C^{2n}$-complementary 
pairs and also $C^{2m^{2}}\delta_{p}A_{r}X\neq\delta_{p}A_{r}X$, hence
$\delta_{p}A_{r}X$ is distinct of $B_{s}X$. Trivially, $\delta_{p}A_{r}X$ is distinct of $D_{t}X$,
$E_{u}X$ and $F_{v}X$. Equally is proved that 
$$C^{2n}\delta_{p}D_{t}X=C^{2n+d_{1}p^{-1}}XC^{2n+d_{1}p^{-1}x}X\cdots C^{2n+d_{t}p^{-1}}XC^{2n+d_{t}p^{-1}x}X.$$
As $p\not\vert4n$, then
\begin{eqnarray*}
p^{-1}(a_{j}-a_{i})&\not\equiv&2n\Mod{4n},\\
p^{-1}(a_{j}-a_{i}x)&\not\equiv&2n\Mod{4n},\\
p^{-1}(a_{j}x-a_{i})&\not\equiv&2n\Mod{4n},\\
p^{-1}x(a_{j}-a_{i})&\not\equiv&2n\Mod{4n},
\end{eqnarray*}
for any pair $(i,j)$. Hence $C^{2n}\delta_{p}D_{t}X\neq\delta_{p}D_{t}X$ and 
$\delta_{p}D_{t}X$ is distinct to $E_{u}X$ and $F_{v}X$. Equally it is shown that $\delta_{p}$
conserves the structure of $B_{s}X$, $E_{u}X$ and $F_{v}X$. The case $x^{q}\equiv1\Mod{4n}$, 
$q$ an odd prime is proved equally.
\end{proof}

Now we are ready for our proof. We will show that $\mathbb{I}_{4nC}(x)$, where $x^{p}\equiv1\mod4n$
and $p$ is a prime, contains no circulant Hadamard matrices by showing that neither
$\mathbb{I}_{4nC}(2n+x)$ nor $\mathbb{I}_{4nC}(x)\setminus\mathbb{I}_{4nC}(2n+x)$ contains
circulant Hadamard matrices. We start with the case $p$ an odd prime

\begin{theorem}\label{theo_case_order_prime}
There is no circulant Hadamard matrices in $\mathbb{I}_{4nC}(x)$ where $x^{p}\equiv1\mod4n$, 
$p$ an odd prime.
\end{theorem}
\begin{proof}
Let $\mu(4n)$ denote the number of solution of 
$$x^{2}\equiv1\Mod{4n},$$
with $n=2^{m-2}\prod_{i}^{d}p_{i}^{h_{i}}$. If exist solution to the equation $x^{p}\equiv1\mod4n$
with $p$ an odd prime, then $\mu(4n)<2\phi(n)$. Thus $\X_{\mathbb{I}_{4n}(x)}$ contains codewords
of lenght $1$ or $p$. From Theorem \ref{theo_principal}, 
$\mathbb{I}_{4n}(xy)\subset\mathbb{I}_{4n}(x)$ for all $y$ such that 
$y^{2}\equiv1\mod4n$, particularly for $y=2n+1$. Thus $z=xy=2n+x$.
Now take a word $Y$ in $\mathbb{I}_{4n}(x)$ such that is not fixed by $\Delta_{4n}$. 
From Theorem \ref{theo_size_orbit}, if $Y_{C}$ is Hadamard, then either $\delta_{p}Y_{C}=Y_{C}$ or
$\delta_{p}Y_{C}=RY_{C}$ for all $p\in\Z_{4n}^{*}$. We will use this and the Lemma
\ref{lemma_2n_subword_prime} and we will show that $Y_{C}$ is no Hadamard. We will keep in mind that
$C^{4n-1}RY$, $C^{n-1}RY$, $C^{2n-1}RY$ and $C^{3n-1}RY$ are in $\mathbb{I}_{4n}(x)$.

\textsc{case $\delta_{z}Y_{C}=Y_{C}$.}\\
This case implies that $\delta_{4n-z}Y_{C}=\delta_{2n-x}Y_{C}=RY_{C}$. We have four cases.
\begin{enumerate}
\item $\delta_{2n-x}(A_{r}X)_{C}=R(A_{r}X)_{C}$.
On the one hand 
\begin{eqnarray*}
\delta_{2n-x}A_{r}X&=&\delta_{2n-x}(C^{a_{1}}XC^{a_{2}}X\cdots C^{a_{r}}X)\\
&=&C^{a_{1}(2n-x^{-1})}XC^{a_{2}(2n-x^{-1})}X\cdots C^{a_{r}(2n-x^{-1})}X\\
&=&C^{2n-x^{-1}a_{1}}XC^{2n-x^{-1}a_{2}}X\cdots C^{2n-x^{-1}a_{r}}X\\
&=&C^{2n-a_{1}}XC^{2n-a_{2}}X\cdots C^{2n-a_{r}}X
\end{eqnarray*}
On the other hand, as $(4n-a_{i})x\equiv(4n-a_{i})\Mod{4n}$, $(n-a_{i})x\equiv(n-a_{i})\Mod{4n}$,
$(2n-a_{i})x\equiv(2n-a_{i})\Mod{4n}$ and $(3n-a_{i})x\equiv(3n-a_{i})\Mod{4n}$, then
\begin{equation}
\delta_{2n-x}A_{r}X=
\begin{cases}
C^{4n-a_{1}}XC^{4n-a_{2}}X\cdots C^{4n-a_{r}}X,\\
C^{n-a_{1}}XC^{n-a_{2}}X\cdots C^{n-a_{r}}X,\text{ if $x\equiv1\Mod{4}$}\\
C^{2n-a_{1}}XC^{2n-a_{2}}X\cdots C^{2n-a_{r}}X,\\
C^{3n-a_{1}}XC^{3n-a_{2}}X\cdots C^{3n-a_{r}}X,\text{ if $x\equiv1\Mod{4}$}
\end{cases}
\end{equation}
Thus must be holded some of following condictions
\begin{eqnarray}
2n-a_{i}&\equiv&(4n-a_{j})\Mod{4n}\\
2n-a_{i}&\equiv&(n-a_{j})\Mod{4n}\\
2n-a_{i}&\equiv&(2n-a_{j})\Mod{4n}\\
2n-a_{i}&\equiv&(3n-a_{j})\Mod{4n}.
\end{eqnarray}
Then we have respectively
\begin{eqnarray}
a_{j}&\equiv&(2n+a_{i})\Mod{4n}\\
a_{j}&\equiv&(a_{i}-n)\Mod{4n}\\
a_{j}&\equiv&(a_{i})\Mod{4n}\\
a_{j}&\equiv&(n+a_{i})\Mod{4n}
\end{eqnarray}
Hence the only possible cases for the definition of $A_{r}X$ are
\begin{eqnarray}
A_{r}X&=&C^{a_{1}}XC^{a_{1}-n}X\cdots C^{a_{r}}XC^{a_{r}-n}X\\
A_{r}X&=&C^{a_{1}}XC^{a_{1}+n}X\cdots C^{a_{r}}XC^{a_{r}+n}X
\end{eqnarray}
and $\vert A_{r}X\vert=2r$. 

\item $\delta_{2n-x}(B_{s}X)_{C}=R(B_{s}X)_{C}$.
On the one hand 
\begin{eqnarray*}
\delta_{2n-x}B_{s}X&=&\delta_{2n-x}(C^{b_{1}}XC^{b_{1}+2n}X\cdots C^{b_{s}}XC^{b_{s}+2n}X)\\
&=&C^{b_{1}(2n-x^{-1})}XC^{(b_{1}+2n)(2n-x^{-1})}X\\
&&\cdots C^{b_{s}(2n-x^{-1})}XC^{(b_{s}+2n)(2n-x^{-1})}X\\
&=&C^{2n-b_{1}}XC^{4n-b_{1}}X\cdots C^{2n-b_{s}}XC^{4n-b_{s}}X
\end{eqnarray*}
On the other hand, as $(4n-b_{i})x\equiv(4n-b_{i})\Mod{4n}$, $(n-b_{i})x\equiv(n-b_{i})\Mod{4n}$,
$(2n-b_{i})x\equiv(2n-b_{i})\Mod{4n}$ and $(3n-b_{i})x\equiv(3n-b_{i})\Mod{4n}$, then
\begin{equation*}
\delta_{2n-x}B_{s}X=
\begin{cases}
C^{4n-b_{1}}XC^{2n-b_{1}}X\cdots C^{4n-b_{s}}XC^{2n-b_{s}}X,\\
C^{n-b_{1}}XC^{3n-b_{1}}X\cdots C^{n-b_{s}}XC^{3n-b_{s}}X,\text{ if $x\equiv1\Mod{4}$}
\end{cases}
\end{equation*}
Thus must be holded some of following condictions
\begin{eqnarray}
2n-b_{i}&\equiv&(4n-b_{j})\Mod{4n}\\
2n-b_{i}&\equiv&(2n-b_{j})\text{ and }4n-b_{k}\equiv(4n-b_{l})\\
2n-b_{i}&\equiv&(n-b_{j})\text{ and }4n-b_{k}\equiv(3n-b_{l})\\
2n-b_{i}&\equiv&(3n-b_{j})\text{ and }4n-b_{k}\equiv(n-b_{l})
\end{eqnarray}
Then we have respectively
\begin{eqnarray}
b_{j}&\equiv&(2n+b_{i})\Mod{4n}\\
b_{j}&\equiv&(b_{i})\Mod{4n}\\
b_{j}&\equiv&(b_{i}-n)\Mod{4n}\\
b_{j}&\equiv&(n+b_{i})\text{ and }b_{l}\equiv(b_{k}-3n)
\end{eqnarray}
Hence the only possible cases for definition of $B_{s}X$ are
\begin{eqnarray}
B_{s}X&=&C^{b_{1}}XC^{b_{1}+2n}XC^{b_{1}-n}XC^{b_{1}+n}X\nonumber\\
&&\cdots C^{b_{s}}XC^{b_{s}+2n}XC^{b_{s}-n}XC^{b_{s}+n}X
\end{eqnarray}
or $B_{s}X$ containing either the subwords $C^{b_{i}}XC^{b_{i}+2n}XC^{b_{i}+n}XC^{b_{i}+3n}X$ 
or the subwords $C^{b_{i}}XC^{b_{i}+2n}XC^{b_{i}-3n}XC^{b_{i}-n}X$. In any case 
$\vert B_{s}X\vert=4s$. 

\item $\delta_{2n-x}(D_{t}X)_{C}=R(D_{t}X)_{C}$.
On the one hand, from Lemma 
\begin{eqnarray*}
\delta_{2n-x}D_{t}X&=&\delta_{2n-x}(\mathsf{C}_{d_{1}}X\cdots\mathsf{C}_{d_{t}}X)\\
&=&\mathsf{C}_{d_{1}(2n-x^{-1})}X\cdots\mathsf{C}_{d_{t}(2n-x^{-1})}X\\
&=&\mathsf{C}_{2n-x^{-1}d_{1}}X\cdots\mathsf{C}_{2n-x^{-1}d_{t}}X\\
&=&\mathsf{C}_{2n-d_{1}}X\cdots\mathsf{C}_{2n-d_{t}}X
\end{eqnarray*}
On the other hand, as $\mathsf{C}_{n-d_{1}}X\cdots\mathsf{C}_{n-d_{t}}X$, 
$\mathsf{C}_{2n-d_{1}}X\cdots\mathsf{C}_{2n-d_{t}}X$,\\ 
$\mathsf{C}_{3n-d_{1}}X\cdots\mathsf{C}_{3n-d_{t}}X$ and
$\mathsf{C}_{4n-d_{1}}X\cdots\mathsf{C}_{4n-d_{t}}X$ is in $R(D_{t}X)_{C}$, then
\begin{equation}
\delta_{2n-x}D_{t}X=
\begin{cases}
\mathsf{C}_{4n-d_{1}}X\cdots\mathsf{C}_{4n-d_{t}}X,\\
\mathsf{C}_{n-d_{1}}X\cdots\mathsf{C}_{n-d_{t}}X,\text{ if $x\equiv1\Mod{4}$}\\
\mathsf{C}_{2n-d_{1}}X\cdots\mathsf{C}_{2n-d_{t}}X,\\
\mathsf{C}_{3n-d_{1}}X\cdots\mathsf{C}_{3n-d_{t}}X,\text{ if $x\equiv1\Mod{4}$}
\end{cases}
\end{equation}
Thus must be holded some of following condictions
\begin{eqnarray}
2n-d_{i}&\equiv&(4n-d_{j})\Mod{4n}\\
2n-d_{i}&\equiv&(n-d_{j})\Mod{4n}\\
2n-d_{i}&\equiv&(2n-d_{j})\Mod{4n}\\
2n-d_{i}&\equiv&(3n-d_{j})\Mod{4n}.
\end{eqnarray}
Then we have respectively
\begin{eqnarray}
d_{j}&\equiv&(2n+d_{i})\Mod{4n}\\
d_{j}&\equiv&(d_{i}-n)\Mod{4n}\\
d_{j}&\equiv&(d_{i})\Mod{4n}\\
d_{j}&\equiv&(n+d_{i})\Mod{4n}
\end{eqnarray}
Hence the only possible cases for definition of $D_{t}X$ are
\begin{eqnarray}
D_{t}X&=&\mathsf{C}_{d_{1}}X\mathsf{C}_{d_{1}-n}X\cdots\mathsf{C}_{d_{t}}X\mathsf{C}_{d_{t}-n}X\\
D_{t}X&=&\mathsf{C}_{d_{1}}X\mathsf{C}_{d_{1}+n}X\cdots\mathsf{C}_{d_{t}}X\mathsf{C}_{d_{t}+n}X
\end{eqnarray}
and $\vert D_{t}X\vert=2pt$. 

\item $\delta_{2n-x}(E_{u}X)_{C}=R(E_{u}X)_{C}$.
On the one hand 
\begin{eqnarray*}
\delta_{2n-x}E_{u}X&=&\delta_{2n-x}(\mathsf{C}_{e_{1}}X\mathsf{C}_{e_{1}+2n}X\cdots\mathsf{C}_{e_{u}}X\mathsf{C}_{e_{u}+2n}X)\\
&=&\mathsf{C}_{e_{1}(2n-x^{-1})}X\mathsf{C}_{(e_{1}+2n)(2n-x^{-1})}X\\
&&\cdots\mathsf{C}_{e_{u}(2n-x^{-1})}X\mathsf{C}_{(e_{u}+2n)(2n-x^{-1})}X\\
&=&\mathsf{C}_{2n-x^{-1}e_{1}}X\mathsf{C}_{4n-x^{-1}e_{1}}X\cdots\mathsf{C}_{2n-x^{-1}e_{u}}X\mathsf{C}_{4n-x^{-1}e_{u}}X\\
&=&\mathsf{C}_{2n-e_{1}}X\mathsf{C}_{4n-e_{1}}X\cdots\mathsf{C}_{2n-e_{u}}X\mathsf{C}_{4n-e_{u}}X
\end{eqnarray*}
On the other hand, as 
$$\mathsf{C}_{4n-e_{1}}X\mathsf{C}_{2n-e_{1}}X\cdots\mathsf{C}_{4n-e_{u}}X\mathsf{C}_{2n-e_{u}}X$$
and
$$\mathsf{C}_{n-e_{1}}X\mathsf{C}_{3n-e_{1}}X\cdots\mathsf{C}_{n-e_{u}}X\mathsf{C}_{3n-e_{u}}X$$
is in $R(E_{u}X)_{C}$, then
\begin{equation*}
\delta_{2n-x}E_{u}X=
\begin{cases}
\mathsf{C}_{4n-e_{1}}X\mathsf{C}_{2n-e_{1}}X\cdots\mathsf{C}_{4n-e_{u}}X\mathsf{C}_{2n-e_{u}}X,\\
\mathsf{C}_{n-e_{1}}X\mathsf{C}_{3n-e_{1}}X\cdots\mathsf{C}_{n-e_{u}}X\mathsf{C}_{3n-e_{u}}X,\text{ if $x\equiv1\Mod{4}$}
\end{cases}
\end{equation*}
Thus must be holded some of following condictions
\begin{eqnarray}
2n-e_{i}&\equiv&(4n-e_{j})\Mod{4n}\\
2n-e_{i}&\equiv&(2n-e_{j})\text{ and }4n-e_{k}\equiv(4n-e_{l})\\
2n-e_{i}&\equiv&(n-e_{j})\text{ and }4n-e_{k}\equiv(3n-e_{l})\\
2n-e_{i}&\equiv&(3n-e_{j})\text{ and }4n-e_{k}\equiv(n-e_{l})
\end{eqnarray}
Then we have respectively
\begin{eqnarray}
e_{j}&\equiv&(2n+e_{i})\Mod{4n}\\
e_{j}&\equiv&(e_{i})\Mod{4n}\\
e_{j}&\equiv&(e_{i}-n)\Mod{4n}\\
e_{j}&\equiv&(n+e_{i})\text{ and }e_{l}\equiv(e_{k}-3n)
\end{eqnarray}
Hence the only possible cases for definition of $E_{u}X$ are
\begin{eqnarray}
E_{s}X&=&\mathsf{C}_{e_{1}}X\mathsf{C}_{e_{1}+2n}X\mathsf{C}_{e_{1}-n}X\mathsf{C}_{e_{1}+n}X\nonumber\\
&&\cdots\mathsf{C}_{e_{s}}X\mathsf{C}_{e_{s}+2n}X\mathsf{C}_{e_{s}-n}X\mathsf{C}_{e_{s}+n}X
\end{eqnarray}
or $E_{u}X$ contains either the subwords 
$\mathsf{C}_{e_{i}}X\mathsf{C}_{e_{i}+2n}X\mathsf{C}_{e_{i}+n}X\mathsf{C}_{e_{i}+3n}X$ 
or the subwords 
$\mathsf{C}_{e_{i}}X\mathsf{C}_{e_{i}+2n}X\mathsf{C}_{e_{i}-3n}X\mathsf{C}_{e_{i}-n}X$. 
In any case $\vert E_{u}X\vert=4pu$. 
\end{enumerate}
Therefore if $Y$ contains the subwords $A_{r}X$, $B_{s}X$, $D_{t}X$, $E_{u}X$, then has length an 
even number. If $Y_{C}$ is Hadamard, then $Y$ must be in $\G_{4m^{2}}(2m^{2}-m)$ with $m$ an 
odd number. But this is no possible since $Y$ has even length.

\textsc{case $\delta_{z}Y_{C}=RY_{C}$.}\\
We have four case.
\begin{enumerate}
\item $\delta_{2n+x}(A_{r}X)_{C}=R(A_{r}X)_{C}$, then
\begin{eqnarray*}
\delta_{2n+x}A_{r}X&=&\delta_{2n+x}(C^{a_{1}}XC^{a_{2}}X\cdots C^{a_{r}}X)\\
&=&C^{a_{1}(2n+x^{-1})}XC^{a_{2}(2n+x^{-1})}X\cdots C^{a_{r}(2n+x^{-1})}X\\
&=&C^{2n+x^{-1}a_{1}}XC^{2n+x^{-1}a_{2}}X\cdots C^{2n+x^{-1}a_{r}}X\\
&=&C^{2n+a_{1}}XC^{2n+a_{2}}X\cdots C^{2n+a_{r}}X
\end{eqnarray*}
as $(4n-a_{i})x\equiv(4n-a_{i})\Mod{4n}$, $(n-a_{i})x\equiv(n-a_{i})\Mod{4n}$,
$(2n-a_{i})x\equiv(2n-a_{i})\Mod{4n}$ and $(3n-a_{i})x\equiv(3n-a_{i})\Mod{4n}$, then
\begin{equation}
\delta_{2n+x}A_{r}X=
\begin{cases}
C^{4n-a_{1}}XC^{4n-a_{2}}X\cdots C^{4n-a_{r}}X,\\
C^{n-a_{1}}XC^{n-a_{2}}X\cdots C^{n-a_{r}}X,\text{ if $x\equiv1\Mod{4}$}\\
C^{2n-a_{1}}XC^{2n-a_{2}}X\cdots C^{2n-a_{r}}X,\\
C^{3n-a_{1}}XC^{3n-a_{2}}X\cdots C^{3n-a_{r}}X,\text{ if $x\equiv1\Mod{4}$}
\end{cases}
\end{equation}
Thus must be holded some of following condictions
\begin{eqnarray}
2n+a_{i}&\equiv&(4n-a_{j})\Mod{4n}\\
2n+a_{i}&\equiv&(n-a_{j})\Mod{4n}\\
2n+a_{i}&\equiv&(2n-a_{j})\Mod{4n}\\
2n+a_{i}&\equiv&(3n-a_{j})\Mod{4n}.
\end{eqnarray}
Then we have respectively
\begin{eqnarray}
a_{j}&\equiv&(2n-a_{i})\Mod{4n}\\
a_{j}&\equiv&(3n-a_{i})\Mod{4n}\\
a_{j}&\equiv&(4n-a_{i})\Mod{4n}\\
a_{j}&\equiv&(n-a_{i})\Mod{4n}
\end{eqnarray}
Hence $A_{r}X$ has the form
\begin{equation}
A_{r}X=
\begin{cases}
C^{a_{1}}XC^{2n-a_{1}}X\cdots C^{a_{r}}XC^{2n-a_{r}}X\\
C^{a_{1}}XC^{3n-a_{1}}X\cdots C^{a_{r}}XC^{3n-a_{r}}X\\
C^{a_{1}}XC^{4n-a_{1}}X\cdots C^{a_{r}}XC^{4n-a_{r}}X\\
C^{a_{1}}XC^{n-a_{1}}X\cdots C^{a_{r}}XC^{n-a_{r}}X
\end{cases}
\end{equation}
and $\vert A_{r}X\vert=2r$. But in the case $i=j$ we have $a_{i}=n$, $2a_{i}=3n$, $a_{i}=2n$
or $2a_{i}=n$. If $n$ is an odd number, the cases 2 and 4 are not hold. Then 
$\vert A_{r}X\vert=2r+1$ in this case.

\item $\delta_{2n+x}(B_{s}X)_{C}=R(B_{s}X)_{C}$. 
Following the above proof, then must be holded some of following condictions
\begin{eqnarray}
2n+b_{i}&\equiv&(4n-b_{j})\Mod{4n}\\
2n+b_{i}&\equiv&(2n-b_{j})\text{ and }4n+b_{k}\equiv(4n-b_{l})\\
2n+b_{i}&\equiv&(n-b_{j})\text{ and }4n+b_{k}\equiv(3n-b_{l})\\
2n+b_{i}&\equiv&(3n-b_{j})\text{ and }4n+b_{k}\equiv(n-b_{l})
\end{eqnarray}
Then we have respectively
\begin{eqnarray}
b_{j}&\equiv&(2n-b_{i})\Mod{4n}\\
b_{j}&\equiv&0\Mod{4n}\\
b_{j}&\equiv&(3n-b_{i})\Mod{4n}\\
b_{j}&\equiv&(n-b_{i})
\end{eqnarray}
Hence the only possible cases for definition of $B_{s}X$ are
\begin{eqnarray}
B_{s}X&=&C^{b_{1}}XC^{b_{1}+2n}XC^{2n-b_{1}}XC^{4n-b_{1}}X\nonumber\\
&&\cdots C^{b_{s}}XC^{b_{s}+2n}XC^{2n-b_{s}}XC^{4n-b_{s}}X
\end{eqnarray}
with $b_{i}\neq0,2n$ or $B_{s}X$ containing the subwords 
$C^{b_{i}}XC^{b_{i}+2n}XC^{3n-b_{i}}XC^{n-b_{i}}X$. In any case $\vert B_{s}X\vert=4s$. If 
$i=j$, then for $n$ an odd number we have $b_{i}=n$ for
some $i$. Then $B_{s}X$ contains the subword $C^{n}XC^{3n}X$ and $\vert B_{s}X\vert=4s+2$.

\item $\delta_{2n+x}(D_{t}X)_{C}=R(D_{t}X)_{C}$. 
For this case must be holded some of following condictions
\begin{eqnarray}
2n+d_{i}&\equiv&(4n-d_{j})\Mod{4n}\\
2n+d_{i}&\equiv&(n-d_{j})\Mod{4n}\\
2n+d_{i}&\equiv&(2n-d_{j})\Mod{4n}\\
2n+d_{i}&\equiv&(3n-d_{j})\Mod{4n}.
\end{eqnarray}
Then we have respectively
\begin{eqnarray}
d_{j}&\equiv&(2n-d_{i})\Mod{4n}\\
d_{j}&\equiv&(3n-d_{i})\Mod{4n}\\
d_{j}&\equiv&0\Mod{4n}\\
d_{j}&\equiv&(n-d_{i})\Mod{4n}
\end{eqnarray}
Hence the only possible cases for definition of $D_{t}X$ are
\begin{eqnarray}
D_{t}X&=&\mathsf{C}_{d_{1}}X\mathsf{C}_{2n-d_{1}}X\cdots\mathsf{C}_{d_{t}}X\mathsf{C}_{2n-d_{t}}X\\
D_{t}X&=&\mathsf{C}_{d_{1}}X\mathsf{C}_{3n-d_{1}}X\cdots\mathsf{C}_{d_{t}}X\mathsf{C}_{3n-d_{t}}X\\
D_{t}X&=&\mathsf{C}_{d_{1}}X\mathsf{C}_{n-d_{1}}X\cdots\mathsf{C}_{d_{t}}X\mathsf{C}_{n-d_{t}}X
\end{eqnarray}
and $\vert D_{t}X\vert=2pt$. In the special case $i=j$ we have $d_{i}=n$ for $n$ an odd number and 
for some $i$. But $nx\equiv n\Mod{4n}$ and the cyclotomic coset $\mathsf{C}_{n}$ has order 1 and
not $p$ as we wish. Then we always have $\vert D_{t}X\vert=2pt$.

\item $\delta_{2n+x}(E_{u}X)_{C}=R(E_{u}X)_{C}$. 
For this case must be holded some of following condictions
\begin{eqnarray}
2n+e_{i}&\equiv&(4n-e_{j})\Mod{4n}\\
2n+e_{i}&\equiv&(2n-e_{j})\text{ and }4n+e_{k}\equiv(4n-e_{l})\\
2n+e_{i}&\equiv&(n-e_{j})\text{ and }4n+e_{k}\equiv(3n-e_{l})\\
2n+e_{i}&\equiv&(3n-e_{j})\text{ and }4n+e_{k}\equiv(n-e_{l})
\end{eqnarray}
Then we have respectively
\begin{eqnarray}
e_{j}&\equiv&(2n-e_{i})\Mod{4n}\\
e_{j}&\equiv&0\Mod{4n}\\
e_{j}&\equiv&(3n-e_{i})\Mod{4n}\\
e_{j}&\equiv&(n-e_{i})
\end{eqnarray}
Hence the only possible cases for definition of $E_{u}X$ are
\begin{eqnarray}
E_{s}X&=&\mathsf{C}_{e_{1}}X\mathsf{C}_{e_{1}+2n}X\mathsf{C}_{2n-e_{1}}X\mathsf{C}_{4n-e_{1}}X\nonumber\\
&&\cdots\mathsf{C}_{e_{s}}X\mathsf{C}_{e_{s}+2n}X\mathsf{C}_{2n-e_{s}}X\mathsf{C}_{4n-e_{s}}X
\end{eqnarray}
or $E_{u}X$ contains the subwords 
$\mathsf{C}_{e_{i}}X\mathsf{C}_{e_{i}+2n}X\mathsf{C}_{3n-e_{i}}X\mathsf{C}_{n-e_{i}}X$ 
In any case $\vert E_{u}X\vert=4pu$. In the special case $i=j$ we have $e_{i}=n$ for $n$ an odd
number and some $i$. For the same reasons given above we always have $\vert E_{u}X\vert=4pu$.
\end{enumerate}

Therefore if the subwords $A_{r}X$, $B_{s}X$, $D_{t}X$ and $E_{u}X$ have length an even number,
then $Y_{C}$ is no Hadamard. Then we must research the case when $Y$ contain $A_{r}X$ with
$\vert A_{r}X\vert=2r+1$. The words containing $A_{r}X$ are
\begin{eqnarray}\label{words_remaining}
&&A_{r}X,A_{r}XB_{s}X,A_{r}XD_{t}X,A_{r}XE_{u}X,A_{r}XB_{s}XD_{t}X,\nonumber\\
&&A_{r}XB_{s}XE_{u}X,A_{r}XD_{t}XE_{u}X,A_{r}XB_{s}XD_{t}XE_{u}X.
\end{eqnarray}

The relation $\delta_{2n+x}Y_{C}=RY_{C}$ implies $\delta_{2n+1}Y_{C}=RY_{C}$. Thus if we want
$\delta_{2n+1}(A_{r}X)_{C}=R(A_{r}X)_{C}$, $\delta_{2n+1}(B_{s}X)_{C}=R(B_{s}X)_{C}$, 
$\delta_{2n+1}(D_{t}X)_{C}=R(D_{t}X)_{C}$ and $\delta_{2n+1}(E_{u}X)_{C}=R(E_{u}X)_{C}$ with
the structure of the subwords already constructed, then the $a_{i}$, $b_{i}$, $d_{i}$ and $e_{i}$ 
can not be even number since 
$2k(2n+1)\equiv2k\Mod{4n}$. If $Y$ is some word in (\ref{words_remaining}), then 
$\vert Y\vert=2m^{2}+m$. But the latter is not possible because there are exactly $2m^{2}$ odd
number in $\Z_{4m^{2}}$. Hence $Y_{C}$ is no Hadamard. Hence there is no circulant Hadamard 
matrices in $\mathbb{I}_{4nC}(x)$.
\end{proof}

Now we will do the proof for the case $x$ having order 2, $x\neq2n+1$

\begin{theorem}\label{theo_case_order_two}
There is no circulant Hadamard matrices in $\mathbb{I}_{4nC}(x)$ where $x^{2}\equiv1\Mod{4n}$ and
$x\neq2n+1$.
\end{theorem}
\begin{proof}
Let $\mu$ be as the previous theorem. By the Sun Ze Theorem 
\begin{equation}
\mu(4n)=
\begin{cases}
2^{d+1},&\text{ if $m=2$}\\
2^{d+2},&\text{ if $m\geq3$}
\end{cases}
\end{equation}
When we compare $\mu(4n)$ and $\phi(4n)$ we have two cases.\\
\textsc{case} $\mu(4n)=\phi(4n)$.\\
If $n=\prod_{i=1}^{d}p_{i}^{h_{i}}$, $p_{i}>2$, then $\mu(4n)=\phi(4n)=2^{g+1}$ and 
\begin{equation}
\prod_{i=1}^{d}p_{i}^{h_{1}-1}(p_{i}-1)=2^{g}
\end{equation}
Therefore $p_{i}^{h_{1}-1}(p_{i}-1)=2$ for all $i$. Hence for $p_{i}>2$ there is not solution.
Now let $n=2^{m-2}$. Then by the Sun Ze Theorem $\mu(2^{m})=4$ and as $\phi(2^{m})=2^{m-1}$,
then $m=3$ and the solution corresponds to $4n=8$. Now let $n=2^{m-2}\prod_{i}^{d}p_{i}^{h_{i}}$, 
$m\geq3$. Then $\mu(4n)=2^{g+2}$ and $\phi(4n)=2^{m-1}\prod_{i=1}^{d}p_{i}^{h_{i}-1}(p_{i}-1)$. So 
$\prod_{i=1}^{d}p_{i}^{h_{i}-1}(p_{i}-1)=2^{g-m+3}$. Thus $h_{i}=1$ and $p_{i}=2^{k_{i}}+1$ for 
all $i$, hence $2^{\sum_{i=1}^{d}k_{i}}=2^{g-m+3}$. As must be $m\geq3$, then 
$g-\sum_{i=1}^{d}k_{i}+3\geq3$. It holds that $\sum_{i=1}^{d}k_{i}\geq1$. Therefore the only 
solution is $g=1$, $k_{1}=1$ and $m=3$ and corresponds to $4n=24$. By exhaustive search in
$\mathbb{I}_{8}(3)$, $\mathbb{I}_{8}(5)$, $\mathbb{I}_{8}(7)$, $\mathbb{I}_{24}(5)$,
$\mathbb{I}_{24}(7)$, $\mathbb{I}_{24}(11)$, $\mathbb{I}_{24}(13)$, $\mathbb{I}_{24}(17)$,
$\mathbb{I}_{24}(19)$ and $\mathbb{I}_{24}(23)$, there is no circulant Hadamard matrices in
$\Z_{2}^{8}$ and $\Z_{2}^{24}$.\\
\textsc{case} $\mu(4n)<\phi(4n)$.\\
Take $y=2n+1$ in $\Z_{4n}^{*}$. Let $z=2n+x$. If $Y_{C}$ in $\mathbb{I}_{4nC}(x)$ is Hadamard, 
then either $\delta_{z}Y_{C}=Y_{C}$ or $\delta_{z}Y_{C}=RY_{C}$.\\
\textsc{case $\delta_{z}Y_{C}=Y_{C}$.}\\
This case implies that $\delta_{2n-x}Y_{C}=RY_{C}$.
From Lemma \ref{lemma_2n_factor}, $Y$ has the subwords $A_{r}X$, $B_{s}X$, $D_{t}X$,
$E_{u}X$ and $F_{v}X$. By using a similar argument to previous theorem it is shown that the length
of previous subwords is an even number.\\
\textsc{case $\delta_{z}Y_{C}=RY_{C}$.}\\
Equally is obtained that $A_{r}X$, $B_{s}X$, $D_{t}X$, $E_{u}X$ and $F_{v}X$ have length an 
even number. If some subwords have odd length, then the decimation $\delta_{2n+1}$ 
guarantees us that all word $Y$ is formed by subwords $C^{g}X$ with $g$ an odd number and 
again $Y_{C}$ is no Hadamard. Hence there is no circulant Hadamard matrices in 
$\mathbb{I}_{4nC}(x)$ and $x\neq2n+1$.
\end{proof}

Finally, we proof for the case $x=2n+1$

\begin{theorem}\label{theo_2n+1}
There is no circulant Hadamard matrices in $\mathbb{I}_{4nC}(2n+1)$.
\end{theorem}
\begin{proof}
The proof is done with $\delta_{2n+x}$ and $x^{p}\equiv1\mod4n$. From Proposition 1 we consider 
the subwords
\begin{eqnarray*}
A_{r}X&=&C^{2a_{1}}XC^{2a_{2}}X\cdots C^{2a_{r}}X,\\
B_{s}X&=&C^{2b_{1}}XC^{2b_{1}+2n}X\cdots C^{2b_{s}}XC^{2b_{s}+2n}X,\\
F_{v}X&=&C^{2f_{1}+1}XC^{2f_{1}+1+2n}X\cdots C^{2f_{v}+1}XC^{2f_{v}+1+2n}X
\end{eqnarray*}
From Theorem \ref{theo_case_order_prime} there is no circulant
Hadamard matrices $Y_{C}$ with $\delta_{2n+x}Y_{C}=Y_{C}$. So we only have the case 
$\delta_{2n+x}Y_{C}=RY_{C}$. This implies that $\delta_{x}Y_{C}=RY_{C}$. Thus we have
\begin{enumerate}
\item $\delta_{x}(A_{r}X)_{C}=R(A_{r}X)_{C}$. Then
\begin{eqnarray*}
\delta_{x}A_{r}X=C^{2a_{1}x^{-1}}XC^{2a_{2}x^{-1}}X\cdots C^{2a_{r}x^{-1}}X
\end{eqnarray*}
and
\begin{equation}
\delta_{x}A_{r}X=
\begin{cases}
C^{4n-2a_{1}}XC^{4n-2a_{2}}X\cdots C^{4n-2a_{r}}X\\
C^{n-2a_{1}}XC^{n-2a_{2}}X\cdots C^{n-2a_{r}}X\\
C^{2n-2a_{1}}XC^{2n-2a_{2}}X\cdots C^{2n-2a_{r}}X\\
C^{3n-2a_{1}}XC^{3n-2a_{2}}X\cdots C^{3n-2a_{r}}X
\end{cases}
\end{equation}
Suppose that the $2a_{i}$ are not fixed by $x$. Thus if the first condiction is fulfilled, then
\begin{equation}
2a_{i}\equiv(4n-2a_{j}x)\Mod{4n}
\end{equation}
and $A_{r}X$ contains subwords of the form $\mathsf{C}_{2a_{i}}X\mathsf{C}_{4n-2a_{i}}X$ where
$$\mathsf{C}_{2a_{i}}=\{2a_{i},2a_{i}x,...,2a_{i}x^{p-1}\}$$
and
$$\mathsf{C}_{4n-2a_{i}}=\{4n-2a_{i},4n-2a_{i}x,...,4n-2a_{i}x^{p-1}\}.$$
But this last implies that $\delta_{x}A_{r}X=A_{r}X$. If the third condiction is fulfilled, then
\begin{equation}
2a_{i}\equiv(2n-2a_{j}x)\Mod{4n}
\end{equation}
and $A_{r}X$ contains subwords of the form $\mathsf{C}_{2a_{i}}X\mathsf{C}_{2n-2a_{i}}X$ where
$$\mathsf{C}_{2n-2a_{i}}=\{2n-2a_{i},2n-2a_{i}x,...,2n-2a_{i}x^{p-1}\}.$$
And again $\delta_{x}A_{r}X=A_{r}X$. The second and fourth cases are not fulfilled if $n$ is an 
odd number.

On the other hand, suppose that some $2a_{i}$ is fixed by $x$. Then $A_{r}X$ contains either the 
subwords $C^{2a_{i}}XC^{4n-2a_{i}}X$ or the subwords $C^{2a_{i}}XC^{2n-2a_{i}}X$. If all $2a_{i}$
in $A_{r}X$ is fixed by $x$, then $\delta_{x}A_{r}X=A_{r}X$. Conversely, if all no $2a_{i}$ is 
fixed by $x$, then $\delta_{x}A_{r}X=A_{r}X$ too. If in $2a_{i}\equiv(4n-2a_{j})\Mod{4n}$ we have
$i=j$, then $2a_{i}=2n$. Thus if $A_{r}X$ contains $C^{2n}X$ and some of the subwords 
$C^{2a_{i}}XC^{4n-2a_{i}}X$, $C^{2a_{i}}XC^{2n-2a_{i}}X$, 
$\mathsf{C}_{2a_{i}}X\mathsf{C}_{4n-2a_{i}}X$, $\mathsf{C}_{2a_{i}}X\mathsf{C}_{2n-2a_{i}}X$, then
$\delta_{x}A_{r}X=A_{r}X$. Therefore, the only $A_{r}X$ fulfilled
$\delta_{x}(A_{r}X)_{C}=R(A_{r}X)_{C}$ is $A_{r}X=C^{2n}X$.

\item $\delta_{x}(B_{s}X)_{C}=R(B_{s}X)_{C}$. For this cases is proved that $B_{s}X$ contain 
the subwords
$$\mathsf{C}_{2b_{i}}X\mathsf{C}_{4n-2b_{i}}X\mathsf{C}_{2n+2b_{i}}X\mathsf{C}_{2n-2b_{i}}X,$$
$$C^{2b_{i}}XC^{4n-2b_{i}}XC^{2n+2b_{i}}XC^{2n-2b_{i}}X,$$
$$XC^{2n}X,$$
Hence $\vert B_{s}X\vert$ is an even number.

\item $\delta_{x}(F_{v}X)_{C}=R(F_{v}X)_{C}$. Equally is proved that $F_{v}X$ contain 
the subwords
$$\mathsf{C}_{2f_{i}+1}X\mathsf{C}_{4n-(2f_{i}+1)}X\mathsf{C}_{2n+(2f_{i}+1)}X\mathsf{C}_{2n-(2f_{i}+1)}X,$$
$$C^{2f_{i}+1}XC^{4n-(2f_{i}+1)}XC^{2n+(2f_{i}+1)}XC^{2n-(2f_{i}+1)}X.$$
Hence $\vert F_{v}X\vert$ is an even number.
\end{enumerate}
Therefore the only words that can be Hadamard are
$$A_{r}XB_{s}X,\ A_{r}XF_{v}X, A_{r}XB_{s}XF_{v}X.$$
As $\vert A_{r}XB_{s}X\vert<2m^{2}+m$ and $\vert A_{r}XF_{v}X\vert<2m^{2}+m$, then the first
possibilities are ruled out. If $Y=A_{r}XB_{s}XF_{v}X$ , then
\begin{eqnarray*}
\vert YC^{2m^{2}}Y\vert&=&\vert A_{r}XC^{2m^{2}}A_{r}X\vert\\
&=&\vert XC^{2m^{2}}X\vert\\
&=&2<2m^{2}
\end{eqnarray*}
for $m>1$. Hence there is no circulant Hadamard matrices in $\mathbb{I}_{4n}(2n+1)$, $n>1$.
\end{proof}

Finally we will obtain the desired theorem

\begin{theorem}
There is no circulant Hadamard matrices in $\Z_{2}^{4n}$.
\end{theorem}
\begin{proof}
We can subdivide $\Z_{2}^{4n}$ in three subsets, namely, the subset of symmetric sequences, 
the subset the sequences of order $8n$ and the subset of sequences of order $8n+4$. Already
was shown that there are no circulant Hadamard matrices in the first two subsets. From
Theorems  \ref{theo_case_order_prime}, \ref{theo_case_order_two} and \ref{theo_2n+1} there is 
no circulant Hadamard matrices in $\mathbb{I}_{(8n+4)C}(x)$ with $x^{p}\equiv1\Mod{(8n+4)}$, 
$p$ a prime. The Theorem \ref{theo_principal} guarantees us that a circulant Hadamard matrix must 
be searched in $\mathbb{I}_{4nC}(x)$ with $x^{p}\equiv1\Mod{4n}$, $p$ a prime. Hence there is 
no circulant Hadamard matrices in $\Z_{2}^{4n}$, for $n>1$.
\end{proof}

\section{Barker conjecture}

Take $A=(a_{0},a_{1},...,a_{n-1})$ in $\Z_{2}^{n}$. Define the aperiodic autocorrelation at shift 
$k$ of $A$ to be
\begin{equation}
C(k)=\sum_{i=1}^{n-k}a_{k}a_{k+i}
\end{equation}
If $\vert C(k)\vert\leq1$ for $k=1,...,n-1$, then $A$ is called a Barker sequence of length $n$.
(For details see [2],[16]) The Barker conjecture asserts that 
\begin{conjecture}
There are no Barker sequences of length $n>13$ .
\end{conjecture}
It is known that the circulant Hadamard matrix conjecture implies the Barker conjecture. Then
\begin{theorem}
There are no Barker sequences of length $n>13$ .
\end{theorem}

\end{document}